\documentclass{article}
\usepackage{amsmath,amssymb,amsthm,mathrsfs}

\newtheorem{thm}{Theorem}[section]
\newtheorem{pro}[thm]{Proposition}
\newtheorem{lem}[thm]{Lemma}

\newtheorem{cor}[thm]{Corollary}

\theoremstyle{definition}

\newtheorem{defi}[thm]{Definition}

\begin{document}
\date{}
\title{\bf The Terwilliger algebra of the doubled Odd graph}

\author{  Lihang Hou$^{\rm a}$\, Suogang Gao$^{\rm b}$\, Na Kang$^{\rm a}$\,\ Bo Hou$^{\rm b,}$\thanks{Corresponding author. E-mail address: houbo1969@163.com.}\\
{\footnotesize  $^{\rm a}$ School of Mathematics and Science, Hebei GEO University, Shijiazhuang, 050031, P. R. China}\\
\footnotesize $^{\rm b}$ School of Mathematical Sciences, Hebei Normal University, Shijiazhuang, 050024, P. R. China}
\maketitle
\begin{abstract}
Let $2.O_{m+1}$ denote the doubled Odd graph with  vertex set $X$ on a set of cardinality $2m+1$, where $m\geq 1$. Fix a vertex $x_0\in X$.
Let $\mathcal{A}:=\mathcal{A}(x_0)$ denote the centralizer algebra
of the stabilizer of $x_0$ in the automorphism group of $2.O_{m+1}$, and
$T:=T(x_0)$ the Terwilliger algebra of $2.O_{m+1}$.
In this paper, we first give a basis of $\mathcal{A}$  by considering
the action of the stabilizer of $x_0$ on $X\times X$ and determine the dimension of $\mathcal{A}$.
Furthermore, we give three subalgebras of  $\mathcal{A}$ such that their direct sum is $\mathcal{A}$ as vector space.
Next, for $m\geq 3$ we find all  isomorphism classes of  irreducible $T$-modules to display  the decomposition of $T$ in a block-diagonalization form. Finally, we show that the two algebras $\mathcal{A}$ and $T$ coincide. This result tells us that the graph $2.O_{m+1}$ may be the first example of bipartite but not $Q$-polynomial distance-transitive graph for which  the corresponding centralizer algebra and  Terwilliger algebra are equal.
\end{abstract}

{\bf \em Key words:} Doubled Odd graph; Terwilliger algebra; Centralizer algebra

{\bf  \em 2010 MSC:} 05C50, 05E15

\section{Introduction}

The Terwilliger algebra of a commutative association scheme was first introduced in \cite{p2}; there it was called the subconstituent algebra.  In general, this algebra is a finite-dimensional, semisimple $\mathbb{C}$-algebra and is non-commutative. The Terwilliger algebra has been successfully used for  studying the commutative association schemes and the distance-regular graphs, in particular the $Q$-polynomial distance-regular graphs.

Let $\Gamma=(Y, E)$ denote a distance-regular graph with vertex set $Y$ and
edge set $E$. Fix a vertex $x\in Y$. It is known that the Terwilliger algebra of $\Gamma$ with respect to $x$
is a subalgebra of the centralizer algebra  of the stabilizer of $x$ in the automorphism group of $\Gamma$. The above two algebras  do not coincide  in general.  However,
it was proved that these two algebras are equal for some distance-regular graphs, for example,
the Hamming graphs, the Johnson graphs,  the ordinary Cycle graphs, the folded $n$-cubes, the halved $n$-cubes and the halved folded $2n$-cubes. An interesting and important problem is to find all distance-regular graphs for which the above two algebras are equal.

The present paper is about the Terwilliger algebra of the doubled Odd graph, and one of its main aims is proving the above two algebras are equal. Before  displaying the main results, let us briefly recall the doubled Odd graph and its folded graph: the Odd graph.
For an integer $m\geq 1$, let $S=\{1,2,\ldots, 2m+1\}$.  Let $X={S\choose m}\cup{S\choose m+1}$, where ${S\choose a}$  denotes the  collection of $a$-subsets of $S$.
The doubled Odd graph on $S$, denoted by $2.O_{m+1}$, is described as the graph
whose vertex set is $X$, where two distinct vertices, say $x,y$,
are adjacent whenever $x\subset y$ or $y\subset x$.
It is easy to see that the path-length distance is given by
\begin{align}\label{eq1}
\partial(x,y)&=|x\cup y-x\cap y|=|x|+|y|-2|x\cap y|\ \ \ \ \ \ \ \  \text{$(x,y\in X)$}.
\end{align}
By folding $2.O_{m+1}$, we can  obtain the Odd graph on $S$.
This graph, denoted by $O_{m+1}$, has vertex set $\mathscr{X}:={S\choose m}$, where two $m$-subsets are adjacent whenever they are disjoint. It is known that $2.O_{m+1}$ is the antipodal $2$-cover of $O_{m+1}$, and that
$2.O_{m+1}$ (resp. $O_{m+1}$) is bipartite but not $Q$-polynomial (resp. almost-bipartite $Q$-polynomial) distance-transitive graph. For more information on the two graphs, we refer to \cite{bcn}.

For $2.O_{m+1}$, fix  the vertex $x_0=\{1,2,\ldots, m\}$ and view it as the base vertex.
Let $\mathcal{A}:=\mathcal{A}(x_0)$ and $T:=T(x_0)$ denote the centralizer algebra
of the stabilizer of $x_0$ in  the automorphism group of $2.O_{m+1}$ and  the  Terwilliger algebra of $2.O_{m+1}$, respectively.

In this paper, we first give a basis of $\mathcal{A}$  by considering
the action of the stabilizer of $x_0$ on $X\times X$ and determine the dimension of $\mathcal{A}$ by computing  the dimension of $\mathscr{A}$ since we find dim($\mathcal{A}$)=4({\rm dim}($\mathscr{A}$)), where $\mathscr{A}:=\mathscr{A}(x_0)$
denotes the corresponding centralizer algebra of $O_{m+1}$ (see Theorem \ref{thm1}).
 Furthermore, we also give three subalgebras of  $\mathcal{A}$ such that their direct sum is $\mathcal{A}$ as vector space (see Proposition \ref{pro3} and Corollary \ref{cor01}).

After displaying some results on $\mathcal{A}$, we then turn to the algebra  $T$.
In \cite{bvc}, B.V.C. Collins investigated the relationship between the Terwilliger algebra of an almost-bipartite distance-regular graph and that of its antipodal $2$-cover.
We know from this paper that $T$ is closely related to  $\mathscr{T}:=\mathscr{T}(x_0)$ which denotes the  Terwilliger algebra of $O_{m+1}$. For the algebra $\mathscr{T}$, J.S. Caughman $et\ al$. \cite{cau} characterized some properties on the irreducible $\mathscr{T}$-modules, and Q. Kong $et\ al$. \cite{klw} determined the  dimension of $\mathscr{T}$.  Based on some results from these papers,
  for $m\geq 3$ we give all  isomorphism classes of  irreducible $\mathscr{T}$-modules and further give all isomorphism classes of  irreducible $T$-modules. Consequently, we  describe a decomposition of $T$ in a block-diagonalization form  by using all the homogeneous components of the standard module $V:=\mathbb{C}^{X}$, each of which is a nonzero subspace of $V$ spanned by the  irreducible $T$-modules that are isomorphic (see Theorem \ref{thm01}); this work is originally motivated by the fact that $T$ is isomorphic to a  direct sum of full matrix algebras. We remark that the decomposition of $T$ in Theorem \ref{thm01} might be useful in coding theory to derive code upper bounds for $2.O_{m+1}$ via semidefinite programming.

Finally, we prove that the two algebras $\mathcal{A}$ and $T$ are  equal and naturally, we obtain a basis of $T$ (see Theorem \ref{thm2}). This result tells us that the graph $2.O_{m+1}$ may be the first example of bipartite but not $Q$-polynomial distance-transitive graph
for which the corresponding centralizer algebra and Terwilliger algebra are equal.

We display some results on  the algebras $\mathscr{A}$ and $\mathscr{T}$ in the Appendix; in particular,  we prove that $\mathscr{A}$ is the same as $\mathscr{T}$ and give a decomposition of $\mathscr{T}$ for $m\geq 3$. These results are important and necessary for our discussions on $2.O_{m+1}$.

We remark that the technique of this paper is as an extension of the approach used in  \cite{hhkg} on the study of the Terwilliger algebra of the halved $n$-cube.

\section{Preliminaries}

In this section, we recall some  concepts and basic facts concerning
distance-regular graphs and Terwilliger algebras.

Let $\mathbb{C}$ denote the complex number field  and let $Y$ denote a nonempty finite set.
Let $V:=\mathbb{C}^{Y}$ denote the column vectors space with coordinates indexed by $Y$. We
endow $V$ with the Hermitian inner product $<, >$ that satisfies $<u,v>=u^{\rm t}\overline{v}$ for $u, v\in V$, where ${\rm t}$ denotes
transpose and ${}^-$ denotes complex conjugation. Let  ${\rm{Mat}}_{Y}(\mathbb{C})$ denote the $\mathbb{C}$-algebra of matrices with rows and columns indexed by $Y$. Observe that ${\rm{Mat}}_{Y}(\mathbb{C})$  naturally acts on $V$ by left multiplication;
we call $V$ the {\it standard module}.

Let $\Gamma=(Y, E)$ denote a finite, undirected, connected graph,
without loops or multiple edges, with path-length distance function $\partial$, and diameter $D:=\text{max}\{\partial(x,y)\mid x,y\in Y\}$. We say $\Gamma$ is {\it{distance-transitive}} if for every $i\ (0\leq i\leq D)$ and all pairs of vertices $(x,y)$ and $(u,v)$ satisfying  $\partial(x,y)=\partial(u,v)=i$, there is an automorphism that maps $x$ to $u$ and $y$ to $v$. We say $\Gamma$ is {\it{distance-regular}} whenever for all integers $i, j,h\ (0\leq i, j,h\leq D)$ and for all vertices $x, y\in Y$
such that $\partial(x, y)=h$, the number $p^h_{ij}:=|\{z\in Y \mid \partial(x, z)=i, \partial(z, y)=j\}|$
is independent of $x$ and $y$. Observe that the property of distance-transitivity implies the property of distance-regularity. Next, we assume $\Gamma$ is distance-regular.

Fix a vertex $x\in Y$ and view it as the ``base vertex". Let ${\rm Aut}_{x}(\Gamma)$ denote the stabilizer of $x$ in the automorphism group of $\Gamma$. For a matrix $M\in {\rm{Mat}}_{Y}(\mathbb{C})$, we say $M$ is invariant under $\sigma\in{\rm Aut}_{x}(\Gamma)$ if $M_{(\sigma y,\sigma z)}=M_{(y,z)}$ for all  $y,z\in Y$.
The {\it{centralizer\ algebra}} of ${\rm Aut}_{x}(\Gamma)$ is  the set of matrices that are invariant under any element of ${\rm Aut}_{x}(\Gamma)$.
In the following, we shall introduce three subalgebras of the centralizer\ algebra of ${\rm Aut}_{x}(\Gamma)$: the Bose-Mesner algebra, the dual Bose-Mesner algebra and the Terwilliger algebra.

For $0\leq i\leq D$, let $A_i\in {\rm{Mat}}_{Y}(\mathbb{C})$ denote the $i$-th {\it{distance matrix}} of $\Gamma$: the $(y,z)$-entry of $A_i$ is $1$ if $\partial(y,z)=i$ and $0$ otherwise.
It is known that the matrices $A_0, A_1, \ldots, A_D$ span
a subalgebra of  ${\rm{Mat}}_{Y}(\mathbb{C})$. This subalgebra is called the
{\it{Bose-Mesner algebra}} of $\Gamma$ and is denoted by $M$. It turns out that $M$ is generated by the {\it{adjacency matrix}} $A_1$.  By \cite[p. 45]{bcn}, $M$ has another basis $E_0, E_1, \ldots, E_D$ called the {\it{primitive idempotents}} of $\Gamma$. We say that $\Gamma$ is  {\it{$Q$-polynomial}}  with respect to a given ordering
$E_0, E_1, \ldots, E_D$ of the primitive idempotents if there are polynomials $q_i$ of degree $i\ (0\leq i\leq D)$ such that $E_i=q_i(E_1)$,
where the matrix multiplication is entrywise.

 For $0\leq i\leq D$, let the diagonal matrix $E^*_i:=E^*_i(x)\in {\rm{Mat}}_{Y}(\mathbb{C})$ denote the $i$-th {\it{dual idempotent}} of $\Gamma$: the
$(y,y)$-entry of $E^*_i$ is $1$ if $\partial(x,y)=i$ and $0$ otherwise.
It is known that $E^*_0, E^*_1, \ldots, E^*_D$ span  a commutative subalgebra $M^*:=M^*(x)$ of ${\rm{Mat}}_{Y}(\mathbb{C})$ and $M^*$ is called the {\it{dual Bose-Mesner algebra}} of $\Gamma$ with respect to $x$.

Let $T:=T(x)$ denote the subalgebra of  ${\rm{Mat}}_{Y}(\mathbb{C})$
generated by $M$ and $M^*$. The algebra $T$ is called
the {\it Terwilliger} (or {\it subconstituent}) {\it algebra} of $\Gamma$ with respect to $x$. This algebra is a
finite-dimensional, semisimple  $\mathbb{C}$-algebra and is non-commutative in general (\cite{p2}).

By a  $T$-{\it{module}}, we mean a subspace $W$ of $V$ such that $BW\subseteq W$
for all $B\in T$.
Let $W,W'$ denote $T$-modules. Then $W,W'$ are said to be $T$-{\it{isomorphic}}  ({\it{isomorphic}} for short) whenever there exists an isomorphism of vector spaces $\phi$: $W\rightarrow W'$ such that
\begin{align*}
(B\phi-\phi B)W=0\ \ \ \text{for all $B\in T$}.
\end{align*}
A $T$-module $W$ is said to be $irreducible$ whenever $W\neq 0$ and $W$ contains no $T$-modules other than  $0$ and $W$. We remark that any two non-isomorphic irreducible $T$-modules are
orthogonal. Furthermore,  every nonzero $T$-module is an orthogonal direct sum of irreducible $T$-modules; in particular, the standard module $V$  is also an orthogonal direct sum of irreducible $T$-modules.

 Let $W$ be an irreducible $T$-module. By the {\it multiplicity} with
which $W$ appears in $V$, we mean the number of irreducible $T$-modules in the above sum which are isomorphic to $W$.
The {\it{diameter}}  (resp. {\it{dual diameter}}) of $W$ is defined as $|\{i\mid0\leq i\leq D, E_i^*W\neq 0\}|-1$ (resp. $|\{i\mid 0\leq i\leq D, E_iW\neq 0\}|-1$). $W$ is said to be {\it{thin}} (resp. {\it{dual thin}}) whenever $\dim(E^*_iW)\leq 1$ (resp. $\dim(E_iW)\leq 1$) for all $0\leq i\leq D$. Note that $W$ is thin (resp. dual thin)  if and only if the diameter (resp. dual diameter)  of $W$ is equal to dim$(W)-1$. By the {\it{endpoint}} of $W$, we mean $\min \{i\mid 0\leq i\leq D, E^*_iW\neq 0\}$.
Next, we suppose $\Gamma$ is $Q$-polynomial with respect to the given ordering
$\{E_i\}^D_{i=0}$.
By the {\it{dual endpoint}} of $W$, we mean  $\min\{i\mid 0\leq i\leq D, E_iW\neq 0\}$).
It is known that $W$ is thin if and only if $W$ is dual thin.

See \cite{p2} for more information on the Terwilliger algebra.

\section{The centralizer algebra of $2.O_{m+1}$}
In this section, we will describe  a concrete basis of the centralizer algebra of $2.O_{m+1}$.
For the rest of this paper, we can choose the vertex $x_0:=\{1,2,\ldots,m\}\in X$ as the base vertex since $2.O_{m+1}$ is distance-transitive.
Let ${\rm Aut}(2.O_{m+1})$ denote the automorphism group of $2.O_{m+1}$ and ${\rm Aut}_{x_0}(2.O_{m+1})$ the stabilizer of $x_0$ in  ${\rm Aut}(2.O_{m+1})$. By \cite[p. 260]{bcn}, ${\rm Aut}(2.O_{m+1})$ is (up to isomorphic) sym$(S)\times\mathbb{Z}_2$, i.e., it permutes
the $2m+1$ elements in $S$ and transforms $x$ to $\bar{x}:=S-x$ for any $x\in X$. Then we  can readily verify that ${\rm Aut}_{x_0}(2.O_{m+1})$ is (up to isomorphic) sym$(x_0)\times\text{sym}(S-x_0)$.

It is clear that a partition of
$X\times X$ is given by the following four subsets:
\begin{align*}
&X_{(m,m)}={S\choose m}\times{S\choose m},\ \ \ \ \ \ \ \ \ \ \ \ \ \ \ \
X_{(m,m+1)}={S\choose m}\times{S\choose m+1},\\
&X_{(m+1,m)}={S\choose m+1}\times{S\choose m},\ \ \ \ \ \ \ \
X_{(m+1,m+1)}={S\choose m+1}\times{S\choose m+1}.
\end{align*}

To each ordered pair $(y,z)\in X\times X$,
we associate the four-tuple $(i,j,t,p)$:
\begin{align}\label{equ2}
\varrho(y,z):=(i,j,t,p),\ \text{where} \ \ i=|x_0\cap y|,\ j=|x_0\cap z|,\  t=|y\cap z|,\  p=|x_0\cap y\cap z|.
\end{align}
Observe that $0\leq i,j,p\leq m$, $0\leq t\leq m+1$ and $0\leq p\leq \text{min}\{i,j,t\}$.

Define $\mathcal{I}_{(m,m)}$ to be the set of all four-tuples $(i, j, t, p)$ that occur as $\varrho(y,z)=(i,j,t,p)$ for some $y,z\in X_{(m,m)}$, that is,
\begin{align}\label{eq2}
\mathcal{I}_{(m,m)}=\big\{(i,j,t,p)\mid \varrho(y,z)=(i,j,t,p),\ \ (y,z)\in  X_{(m,m)}\big\}.
\end{align}
Similarly, we may define
\begin{align}
&\mathcal{I}_{(m,m+1)}=\big\{(i,j,t,p)\mid\varrho(y,z)=(i,j,t,p),\ \ (y,z)\in  X_{(m,m+1)}\big\},\label{eq3}\\
&\mathcal{I}_{(m+1,m)}=\big\{(i,j,t,p)\mid\varrho(y,z)=(i,j,t,p),\ \ (y,z)\in  X_{(m+1,m)}\big\}\nonumber \\
\text{\ and}\ \ \ \ &\mathcal{I}_{(m+1,m+1)}=\big\{(i,j,t,p)\mid\varrho(y,z)=(i,j,t,p),\ \ (y,z)\in  X_{(m+1,m+1)}\big\}.\ \ \ \ \ \ \ \ \ \ \ \ \ \ \ \ \ \ \nonumber
\end{align}

\begin{pro}\label{pro1}
 For $m\geq 1$, the following {\rm (i)--(iv)} hold.
\begin{itemize}
\item[\rm (i)] The set $\mathcal{I}_{(m,m)}$ is
\begin{align}\label{eq6}
\big\{(i,j,t,p)\mid\ &0\leq i,j\leq m,\  \mathrm{max}\{i+j-m,m-1-i-j\}\leq t\leq m-|i-j|, \nonumber\\
&\mathrm{max}\{0,i+j-m,i+t-m,j+t-m\}\leq p\leq \\
&\ \ \ \ \ \ \ \ \ \  \ \ \ \ \ \ \  \ \ \ \ \ \ \ \ \ \  \ \ \ \ \ \ \ \ \ \ \ \ \ \ \ \ \ \ \ \ \ \ \ \ \mathrm{min}\{i,j,t,i+j+t+1-m\}\big\}.\nonumber
\end{align}
Moreover, the cardinality of $\mathcal{I}_{(m,m)}$ is
${m+4\choose 4}$.
\item[\rm (ii)] The set $\mathcal{I}_{(m,m+1)}$ is
\begin{align}\label{eq7}
\big\{(i,j,t,p)\mid\ &0\leq i,j\leq m,\  |i+j-m|\leq t\leq m-\mathrm{max}\{i-j,j-i-1\}, \nonumber\\
& i-\mathrm{min}\{i,m-j,m-t,i-j-t+m+1\}\leq p\leq \\
&\ \ \ \ \ \ \ \ \ \  \ \ \ \ \ \ \  \ \ \ \ \ \ \ \ \ \  \ \ \ \ \ \ \ \ \ \ \ \ \ \ \ \ \ \ \ \  i-\mathrm{max}\{0,i-j,i-t,m-j-t\} \big\}.\nonumber
\end{align}
Moreover, the cardinality of $\mathcal{I}_{(m,m+1)}$ is
${m+4\choose 4}$.
\item[\rm (iii)] The set $\mathcal{I}_{(m+1,m)}$ is
\begin{align*}
\big\{(i,j,t,p)\mid\ &0\leq i,j\leq m,\  |i+j-m|\leq t\leq m-\mathrm{max}\{j-i,i-j-1\}, \nonumber\\
& j-\mathrm{min}\{m-i,j,m-t,j-i-t+m+1\}\leq p\leq \\
&\ \ \ \ \ \ \ \ \ \  \ \ \ \ \ \ \  \ \ \ \ \ \ \ \ \ \  \ \ \ \ \ \ \ \ \ \ \ \ \ \ \ \ \ \ \ \  j-\mathrm{max}\{0,j-i,j-t,m-i-t\} \big\}.\nonumber
\end{align*}
Moreover, the cardinality of $\mathcal{I}_{(m+1,m)}$ is
${m+4\choose 4}$.
\item[\rm (iv)] The set $\mathcal{I}_{(m+1,m+1)}$ is
\begin{align*}
\big\{(i,j,t,p)\mid\ &0\leq i,j\leq m,\  1+\mathrm{max}\{i+j-m-1,m-i-j\}\leq t\leq m+1-|i-j|, \nonumber\\
&i+j-m+\mathrm{max}\{0,m-i-j,t-i-1,t-j-1\}\leq p\leq i+j-m+\\
&\ \ \ \ \ \ \ \ \ \  \ \ \ \ \ \ \  \ \ \ \ \ \ \ \ \ \  \ \ \ \ \ \ \ \ \ \ \ \ \mathrm{min}\{m-i,m-j,t-1,m-i-j+t\}\big\}.\nonumber
\end{align*}
Moreover, the cardinality of $\mathcal{I}_{(m+1,m+1)}$ is
${m+4\choose 4}$.
\end{itemize}
\end{pro}
\begin{proof}
(i) Immediate from Proposition \ref{pro5} in the Appendix and the fact that  the set $\mathcal{I}_{(m,m)}$ is the same as the set $\mathcal{I}_{m}$ in \eqref{eq25} for the Odd graph $O_{m+1}$.

(ii) To give \eqref{eq7}, we first define a map from $X_{(m,m+1)}$ to $X_{(m,m)}$ by
\begin{align}\label{eq10}
\ \ \ \ \  \ \ \ \ \  \ \ \ \ \  \  \ \  (y,z)\rightarrow(y,\bar{z}), \ \ \text{where $\bar{z}=S-z$.}
\end{align}
 Clearly, the above map is a bijection. For any given $(y,z)\in X_{(m,m+1)}$, we associate  the four-tuple $(i,j,t,p)\in \mathcal{I}_{(m,m+1)}$:  $\varrho(y,z)=(i,j,t,p)$. Under this bijection, it is easy to verify that the corresponding $(y,\bar{z})$ is associated with the four-tuple $(i,m-j,m-t,i-p)\in \mathcal{I}_{(m,m)}$: $\varrho(y,\bar{z})=(i,m-j,m-t,i-p)$. Therefore, the bijection \eqref{eq10} naturally induces the following bijection from $\mathcal{I}_{(m,m+1)}$ to $\mathcal{I}_{(m,m)}$:
\begin{align}
\ \ \ \ \  \ \ \ \ \  \ \ \ \ \  \ \ \ (i,j,t,p)\rightarrow (i,m-j,m-t,i-p).
\end{align}
This immediately implies that  the cardinality of $\mathcal{I}_{(m,m+1)}$ is equal to that of $\mathcal{I}_{(m,m)}$, i.e., $|\mathcal{I}_{(m,m+1)}|={m+4\choose 4}$. Since the four-tuple $(i,m-j,m-t,i-p)\in \mathcal{I}_{(m,m)}$, then we can easily obtain  \eqref{eq7} by substituting  this four-tuple into \eqref{eq6} and simplifying the result.

(iii) Use a similar argument used in the proof of (ii). The map from $X_{(m+1,m)}$ to $X_{(m,m)}$: $(y,z)\rightarrow(\bar{y},z)$,
naturally induces a bijection from $\mathcal{I}_{(m+1,m)}$ to $\mathcal{I}_{(m,m)}$:\\
\text{\ \ \ \ \ \ \ \ \ \ \ \ \ \ \ \ \ \ \ \  \ \ \ \ \ \ \ \ \ \  \ \ \ \ \ \ \ \ \ \ \ \ \ \ \ }$(i,j,t,p)\rightarrow (m-i,j,m-t,j-p)$.

(iv) Use a similar argument used in the proof of (ii). The map from $X_{(m+1,m+1)}$ to $X_{(m,m)}$: $(y,z)\rightarrow(\bar{y},\bar{z})$,
naturally induces a bijection from $\mathcal{I}_{(m+1,m+1)}$ to $\mathcal{I}_{(m,m)}$:\\
\text{\ \ \ \ \ \ \ \ \ \ \ \ \ \ \ \ \ \ \ \  \ \ \ \ \ \ \ \ \ \  \ \ \ \ \ \ \ \ \ \ \ \ \  } $(i,j,t,p)\rightarrow (m-i,m-j,t-1,m-i-j+p)$.
\end{proof}

For each $(i,j,t,p)\in \mathcal{I}_{(m,m)}$, we further define the associated set
\begin{align}\label{equ1}
&\ \ \ \ \ \ \ \ \ \ \ \ \ \ \ \ \ \ X^{(i,j,t,p)}_{(m,m)}=\{(y,z)\in  X_{(m,m)}\mid \varrho(y,z)=(i,j,t,p)\}.\
\end{align}
Similarly, for each $(i,j,t,p)\in \mathcal{I}_{(m,m+1)}$, we define
\begin{align*}
&\ \ \ \ \ \ \ \ \ \ \ \ \ \ \ \ \ \ X^{(i,j,t,p)}_{(m,m+1)}=\{(y,z)\in X_{(m,m+1)}\mid \varrho(y,z)=(i,j,t,p)\};
\end{align*}
for each $(i,j,t,p)\in \mathcal{I}_{(m+1,m)}$, we define
\begin{align*}
&\ \ \ \ \ \ \ \ \ \ \ \ \ \ \ \ \ \ X^{(i,j,t,p)}_{(m+1,m)}=\{(y,z)\in X_{(m+1,m)}\mid \varrho(y,z)=(i,j,t,p) \};
\end{align*}
for each $(i,j,t,p)\in \mathcal{I}_{(m+1,m+1)}$, we define
\begin{align*}
&\ \ \ \ \ \ \ \ \ \ \ \ \ \ \ \ \ \  X^{(i,j,t,p)}_{(m+1,m+1)}=\{(y,z)\in X_{(m+1,m+1)}\mid \varrho(y,z)=(i,j,t,p)\}.
\end{align*}

Observe that the subsets $X^{(i,j,t,p)}_{(m,m)}$, $(i,j,t,p)\in \mathcal{I}_{(m,m)}$,
give a partition of  $X_{(m,m)}$. The similar result applies to $X^{(i,j,t,p)}_{(m,m+1)}$, $X^{(i,j,t,p)}_{(m+1,m)}$ and $X^{(i,j,t,p)}_{(m+1,m+1)}$. Therefore, all these subsets give a partition of $X\times X$.
Below, we  give their respective meanings.
\begin{pro}\label{pro2}
With notation as above, the following {\rm (i)--(iv)} hold.
\begin{itemize}
\item[\rm (i)] The sets $X^{(i,j,t,p)}_{(m,m)}$, $(i,j,t,p)\in \mathcal{I}_{(m,m)}$, are the orbits of $X_{(m,m)}$ under the action of ${\rm Aut}_{x_0}(2.O_{m+1})$.
\item[\rm (ii)] The sets $X^{(i,j,t,p)}_{(m,m+1)}$, $(i,j,t,p)\in \mathcal{I}_{(m,m+1)}$, are the orbits of $X_{(m,m+1)}$ under the action of ${\rm Aut}_{x_0}(2.O_{m+1})$.
\item[\rm (iii)] The sets $X^{(i,j,t,p)}_{(m+1,m)}$, $(i,j,t,p)\in \mathcal{I}_{(m+1,m)}$, are the orbits of $X_{(m+1,m)}$ under the action of ${\rm Aut}_{x_0}(2.O_{m+1})$.
\item[\rm (iv)] The sets $X^{(i,j,t,p)}_{(m+1,m+1)}$, $(i,j,t,p)\in \mathcal{I}_{(m+1,m+1)}$, are the orbits of $X_{(m+1,m+1)}$ under the action of ${\rm Aut}_{x_0}(2.O_{m+1})$.
\end{itemize}
\end{pro}

\begin{proof}
(i) Immediate from Proposition \ref{pro6} in the Appendix and the observation that   ${\rm Aut}_{x_0}(2.O_{m+1})$ is the same as
${\rm Aut}_{x_0}(O_{m+1})$ and  $X_{(m,m)}=\mathscr{X}\times \mathscr{X}$. 

(ii)--(iv)  Use a similar argument used in the proof of Proposition \ref{pro6}.
\end{proof}
For convenience, we call $X^{(i,j,t,p)}_{(m,m)},\ X^{(i,j,t,p)}_{(m,m+1)},\ X^{(i,j,t,p)}_{(m+1,m)}$ and $X^{(i,j,t,p)}_{(m+1,m+1)}$ the orbits of  $X\times X$ under the action of ${\rm Aut}_{x_0}(2.O_{m+1})$ of type I, type II, type III and type IV, respectively.

\begin{cor}\label{cor1}
The collection of all the orbits of four types  gives the orbits
of $X\times X$ under the action
of {\rm ${\rm Aut}_{x_0}(2.O_{m+1})$}.
\end{cor}
\begin{proof}
Immediate from Proposition \ref{pro2}.
\end{proof}

In the following, we define some useful matrices  of 0s and 1s in ${\rm{Mat}}_X(\mathbb{C})$ depending on
the type of orbits of $X\times X$ under the action of ${\rm Aut}_{x_0}(2.O_{m+1})$.\\

\noindent
$\bullet$ The matrices of type I: For each $(i,j,t,p)\in \mathcal{I}_{(m,m)}$, define the matrix $\mathcal{M}^{t,p}_{i,j}\in {\rm{Mat}}_X(\mathbb{C})$ by
\begin{equation}\label{eq13}
(\mathcal{M}^{t,p}_{i,j})_{yz}=\left\{\begin{array}{ll} 1 &\text{if}\ (y,z)\in X^{(i,j,t,p)}_{(m,m)},\\[0.2cm]
 0 &\text{otherwise } \end{array}\right.
\ \ (y,z\in X).
\end{equation}

\noindent
$\bullet$ The matrices of type II: For each $(i,j,t,p)\in \mathcal{I}_{(m,m+1)}$, define the matrix $\mathcal{R}^{t,p}_{i,j}\in {\rm{Mat}}_X(\mathbb{C})$ by
\begin{equation}
(\mathcal{R}^{t,p}_{i,j})_{yz}=\left\{\begin{array}{ll} 1 &\text{if}\ (y,z)\in X^{(i,j,t,p)}_{(m,m+1)},\\[0.2cm]
 0 &\text{otherwise } \end{array}\right.
\ \ (y,z\in X).
\end{equation}

\noindent
$\bullet$ The matrices of type III: For each $(i,j,t,p)\in \mathcal{I}_{(m+1,m)}$, define the matrix $\mathcal{L}^{t,p}_{i,j}\in {\rm{Mat}}_X(\mathbb{C})$ by
\begin{equation}
(\mathcal{L}^{t,p}_{i,j})_{yz}=\left\{\begin{array}{ll} 1 &\text{if}\ (y,z)\in X^{(i,j,t,p)}_{(m+1,m)},\\[0.2cm]
 0 &\text{otherwise } \end{array}\right.
\ \ (y,z\in X).
\end{equation}

\noindent
$\bullet$ The matrices of type IV: For each $(i,j,t,p)\in \mathcal{I}_{(m+1,m+1)}$, define the matrix $\mathcal{F}^{t,p}_{i,j}\in {\rm{Mat}}_X(\mathbb{C})$ by
\begin{equation}\label{eq23}
(\mathcal{F}^{t,p}_{i,j})_{yz}=\left\{\begin{array}{ll} 1 &\text{if}\ (y,z)\in X^{(i,j,t,p)}_{(m+1,m+1)},\\[0.2cm]
 0 &\text{otherwise } \end{array}\right.
\ \ (y,z\in X).
\end{equation}

For each  $\mathcal{M}^{t,p}_{i,j},\ (i,j,t,p)\in \mathcal{I}_{(m,m)}$, observe  its   transpose is $\mathcal{M}^{t,p}_{j,i}$ and it is invariant under permutating the rows and columns by elements of  ${\rm Aut}_{x_0}(2.O_{m+1})$ by Corollary \ref{cor1}. The similar results can apply to the matrices of other three types; but note that the transpose of $\mathcal{R}^{t,p}_{i,j}$ is $\mathcal{L}^{t,p}_{j,i}$. Moreover, it is easy to see that all the matrices of four types are linearly independent. Let $\mathcal{A}:=\mathcal{A}(x_0)$
be the linear space over $\mathbb{C}$ spanned by all these matrices. It is known that $\mathcal{A}$ is  the centralizer  algebra of ${\rm Aut}_{x_0}(2.O_{m+1})$.

\begin{thm}\label{thm1}
For $m\geq 1$, all the  matrices of four types
give  a basis of  $\mathcal{A}$ with
\begin{align}\label{eq16}
{\rm dim}(\mathcal{A})=4{m+4\choose 4}.
\end{align}
\end{thm}
\begin{proof}
Immediate from the above discussions.
\end{proof}

Let $\mathcal{A}_1$, $\mathcal{A}_2$ and $\mathcal{A}_3$ be the linear subspace of $\mathcal{A}$ spanned by the matrices  of type I, of types II, III, and of type IV respectively.

\begin{pro}\label{pro3}
The following {\rm (i)--(iii)} hold.
\begin{itemize}
\item[\rm (i)] The subspace $\mathcal{A}_1$ is a subalgebra of $\mathcal{A}$ with a basis of matrices  of type I.
\item[\rm (ii)] The subspace $\mathcal{A}_2$ is a subalgebra of $\mathcal{A}$ with a basis of matrices  of types II,III.
\item[\rm (iii)] The subspace $\mathcal{A}_3$ is a subalgebra of $\mathcal{A}$ with a basis of matrices  of type IV.
\end{itemize}
\end{pro}
\begin{proof}
(i) It is easy to check that $\mathcal{A}_1$ is closed under addition, scalar multiplication,  taking the conjugate
transpose and matrix multiplication. Therefore,  $\mathcal{A}_1$ is a subalgebra of $\mathcal{A}$.

(ii)--(iii) Similar to the proof of (i).
\end{proof}
\begin{cor}\label{cor01}
We have
\begin{align}
\mathcal{A}=\mathcal{A}_1+\mathcal{A}_2+\mathcal{A}_3
 \ \ \text{\rm (direct sum of vector spaces)}.
\end{align}
\end{cor}
\begin{proof}
Immediate from Theorem \ref{thm1} and Proposition \ref{pro3}.
\end{proof}

\section{The Terwilliger algebra of $2.O_{m+1}$}
For $2.O_{m+1}$, let $A_1$ and $E^*_i:=E^*_i(x_0)\ (0\leq i\leq 2m+1)$ denote its  adjacency  matrix and the  $i$-th dual idempotent, respectively. Let $T:=T(x_0)$ denote its Terwilliger algebra generated by the matrices $A_1,E^*_0,E^*_1,\ldots, E^*_{2m+1}$.
In this section, we first show that $T$ is  a subalgebra of {\rm $\mathcal{A}$}. Then for $m\geq 3$, we give the decomposition of $T$ in a block-diagonalization form. Finally, we prove that $\mathcal{A}$ coincides with $T$. We begin with the following lemma.
\begin{lem}\label{lem2} The following {\rm (i)--(iv)} hold.
\begin{itemize}
\item[\rm (i)] For each $0\leq i\leq 2m+1$,
\begin{equation}\label{eq5}
E^*_i=\left\{\begin{array}{ll} \mathcal{M}^{m,\frac{2m-i}{2}}_{\frac{2m-i}{2},\frac{2m-i}{2}} &\text{if $i$ is even},\\[.3cm]
\mathcal{F}^{m+1,\frac{2m+1-i}{2}}_{\frac{2m+1-i}{2},\frac{2m+1-i}{2}} &\text{if $i$ is odd}.
\end{array}\right.
\end{equation}
\item[\rm (ii)] For each $0\leq i\leq 2m$,
\begin{equation}\label{eq12}
E^*_iA_1E^*_{i+1}=\left\{\begin{array}{ll} \mathcal{R}^{m,\frac{2m-i}{2}}_{\frac{2m-i}{2},\frac{2m-i}{2}} &\text{if $i$ is even},\\[.3cm]
 \mathcal{L}^{m,\frac{2m-1-i}{2}}_{\frac{2m+1-i}{2},\frac{2m-1-i}{2}} &\text{if $i$ is odd}.
\end{array}\right.
\end{equation}
\item[\rm (iii)] For each $0\leq i\leq 2m$,
\begin{equation}\label{eq05}
E^*_{i+1}A_1E^*_i=\left\{\begin{array}{ll} \mathcal{L}^{m,\frac{2m-i}{2}}_{\frac{2m-i}{2},\frac{2m-i}{2}} &\text{if $i$ is even},\\[.3cm]
 \mathcal{R}^{m,\frac{2m-1-i}{2}}_{\frac{2m-1-i}{2},\frac{2m+1-i}{2}} &\text{if $i$ is odd}.
\end{array}\right.
\end{equation}
\item[\rm (iv)]
\begin{align*}
A_1=\sum^{2m}_{\stackrel{i=0}{i\ even}}(\mathcal{R}^{m,\frac{2m-i}{2}}_{\frac{2m-i}{2},\frac{2m-i}{2}}+\mathcal{L}^{m,\frac{2m-i}{2}}_{\frac{2m-i}{2},\frac{2m-i}{2}})+
\sum^{2m}_{\stackrel{i=0}{i\ odd}}(\mathcal{R}^{m,\frac{2m-1-i}{2}}_{\frac{2m-1-i}{2},\frac{2m+1-i}{2}}+\mathcal{L}^{m,\frac{2m-1-i}{2}}_{\frac{2m+1-i}{2},\frac{2m-1-i}{2}})
\end{align*}
\end{itemize}
\end{lem}

\begin{proof}
(i) For even $i\ (0\leq i\leq 2m+1)$ and for $y,z\in X$, we consider the $(y,z)$-entries of matrices at both sides of \eqref{eq5}.
By \eqref{eq1} and \eqref{eq13},
it is easy to see that
$(E^*_i)_{yz}=(\mathcal{M}^{m,\frac{2m-i}{2}}_{\frac{2m-i}{2},\frac{2m-i}{2}})_{yz}=1$ if $y=z,\ |y|=m,\ |x_0\cap y |=\frac{2m-i}{2}$, and 0 otherwise.  This means that \eqref{eq5}  holds for even $i$. Similarly,  for odd $i\ (0\leq i\leq 2m+1)$ and for $y,z\in X$,
we have $(E^*_i)_{yz}=(\mathcal{F}^{m+1,\frac{2m+1-i}{2}}_{\frac{2m+1-i}{2},\frac{2m+1-i}{2}})_{yz}=1$ if $y=z,\ |y|=m+1,\ |x_0\cap y |=\frac{2m+1-i}{2}$, and 0 otherwise. This means that \eqref{eq5}  also holds for odd $i$.

(ii) Similar to the proof of (i): for even $i\ (0\leq i\leq 2m)$ and for $y,z\in X$, we have
$(E^*_{i}A_1E^*_{i+1})_{yz}=(\mathcal{R}^{m,\frac{2m-i}{2}}_{\frac{2m-i}{2},\frac{2m-i}{2}})_{yz}=1$ if $|y|=m,\ |z|=m+1,\ |x_0\cap y |=|x_0\cap z|=\frac{2m-i}{2}$, $y\subset z$, and 0 otherwise; for odd $i\ (0\leq i\leq 2m)$ and for $y,z\in X$, we have
$(E^*_{i}A_1E^*_{i+1})_{yz}=(\mathcal{L}^{m,\frac{2m-1-i}{2}}_{\frac{2m+1-i}{2},\frac{2m-1-i}{2}})_{yz}=1$ if $|y|=m+1,\ |z|=m,\ |x_0\cap y|=\frac{2m+1-i}{2}$, $|x_0\cap z|=\frac{2m-1-i}{2}$, $z\subset y$, and 0 otherwise.

(iii) Take transpose of the matrices at both sides of \eqref{eq12}.

(iv) We have
\begin{align}\label{eq21}
A_1&=\big(\sum^{2m+1}_{i=0}E^*_i\big)A_1\big(\sum^{2m+1}_{i=0}E^*_i\big)=\sum^{2m}_{i=0}(E^*_{i}A_1E^*_{i+1}+E^*_{i+1}A_1E^*_{i}),
\end{align}
where the second equality holds since $2.O_{m+1}$ is bipartite (i.e. $\sum^{2m+1}_{i=0}E^*_{i}A_1E^*_{i}=0$). From \eqref{eq12}, \eqref{eq05} and \eqref{eq21}, the result follows.
\end{proof}

\begin{lem}\label{lem3}
The algebra $T$ is a subalgebra of $\mathcal{A}$.
\end{lem}
\begin{proof}
 Immediate from Lemma \ref{lem2}(i),(iv).
\end{proof}

 We assume $m\geq 3$ for the rest of this section.
Below, we aim to give the decomposition of $T$ in a block-diagonalization form. To this end, we need characterize  some properties of irreducible $T$-modules.

In \cite{bvc}, B.V.C. Collins investigated the relationship between the Terwilliger algebra of an almost-bipartite distance-regular graph and that of its antipodal $2$-cover. Since $2.O_{m+1}$ is the antipodal $2$-cover of $O_{m+1}$ and  $O_{m+1}$ is almost-bipartite, then we from \cite{bvc} know  that the irreducible modules of Terwilliger algebra for $2.O_{m+1}$ are closely related to that for $O_{m+1}$. For this reason, we first describe some results on the irreducible modules of Terwilliger algebra for $O_{m+1}$.

Let  $\mathscr{V}:=\mathbb{C}^{\mathscr{X}}$ denote the standard module, and let $\mathscr{T}:=\mathscr{T}(x_0)$ denote the corresponding Terwilliger algebra of $O_{m+1}$.  Some properties  of irreducible $\mathscr{T}$-modules are  characterized in the papers \cite{cau,bvc} (see the following lemma).
\begin{lem}\label{lem4}
Let $\mathscr{W}$ denote an irreducible $\mathscr{T}$-module and
let $\nu, \mu, d\ (0\leq \nu, \mu,  d\leq m) $ denote its endpoint, dual endpoint and diameter, respectively. Then the following  {\rm{(i)--(iv)}} hold.
\begin{itemize}
\item[\rm(i)] $\nu+d=m$.
\item[\rm(ii)] $\mathscr{W}$ is thin and dual thin.
\item[\rm(iii)] The isomorphism class of $\mathscr{W}$  depends only on the pair $(\mu, d)$  restricted to
\begin{align}\label{eq15}
 \Upsilon:=\{(\mu,d)\in \mathbb{Z}^2|\ 0\leq d\leq m,\ \frac{1}{2}(m-d)\leq\mu\leq m-d \}.
\end{align}
\item[\rm(iv)] There exist some irreducible $\mathscr{T}$-modules with dual endpoint $\mu$ and diameter $d$ if and only if  $(\mu,d)\in  \Upsilon$.
\end{itemize}
\end{lem}
Note that the above (i), (ii) and (iii)  are from \cite[Theorem 14.3]{bvc}, \cite[Lemma 10.3]{cau} and \cite[Section 16]{cau}, respectively. We remark here that the proof of (iv) shall be given in the Appendix.

 Recall the definition of $2.O_{m+1}$ and of $O_{m+1}$. It is easy to verify that there is a natural $2$-to-$1$ surjection $\pi$: $X\rightarrow\mathscr{X}$  defined by $\pi(x)=\pi(\bar{x})=x$, where $x\in \mathscr{X}$ and $\bar{x}=S-x$. And moreover this surjection $\pi$ preserves adjacency. Therefore, $2.O_{m+1}$ is the antipodal $2$-cover of $O_{m+1}$.
The following definition is from \cite{bvc} and is useful in describing the relationships between   $T$-modules and $\mathscr{T}$-modules.
\begin{defi}{\rm (\cite{bvc})}\label{def}
Let the map $\pi$ be as above. Define  the matrix $\psi:=\psi_{\pi}$, with rows indexed by  $\mathscr{X}$ and columns indexed by $X$, by
\begin{equation*}
(\psi)_{yz}=\left\{\begin{array}{ll} 1 &\text{if}\ \pi z=y,\\
 0 &\text{otherwise } \end{array}\right.
\ \ (y\in \mathscr{X},\ z\in X).
\end{equation*}
\end{defi}
We now turn to $T$-module and give some results on  $T$-modules.
\begin{lem}{\rm (\cite{bvc})}\label{lem5}
Let $W$ denote a $T$-module and let  $\psi$ be the matrix in {\rm Definition \ref{def}}. Then the following  {\rm{(i)--(iv)}} hold.
\begin{itemize}
 \item[\rm(i)] $\psi W$ is a $\mathscr{T}$-module. Moreover, the map $W\rightarrow \psi W$ from the set of all $T$-modules to the set of all $\mathscr{T}$-modules is  a bijection.
\item[\rm(ii)] Let $W,W'$ be $T$-modules. Then $W$ and $W'$ are isomorphic $T$-modules if
and only if   $\psi W$ and $\psi W'$ are isomorphic $\mathscr{T}$-modules.
\item[\rm(iii)] $T$-module $W$ is irreducible  if and only if $\mathscr{T}$-module $\psi W$ is irreducible. Moreover,  $T$-module $W$ is  thin, irreducible if and only if $\mathscr{T}$-module $\psi W$ is  thin, irreducible.
\item[\rm(iv)] Let $W$ be an irreducible  $T$-module. Then $W$ and  $\psi W$ have the same endpoint $\nu\ (0\leq \nu\leq m)$, and the diameters of $W$ and $\psi W$ are $2m+1-2\nu$ and $m-\nu$, respectively.
\end{itemize}
\end{lem}
Note that the above (i), (ii), (iii) and (iv) are from \cite[Corollary 11.2]{bvc}, \cite[Theorem 12.3]{bvc}, \cite[Theorems 11.4, 15.1]{bvc} and \cite[Lemmas 14.1, 14.3]{bvc}, respectively.

Let $W$ denote an irreducible $T$-module and $\psi W$ the corresponding irreducible $\mathscr{T}$-module having dual endpoint $\mu$ and diameter $d$, where $(\mu, d)\in \Upsilon$ by Lemma \ref{lem4}(iii). In this case, we may say $W$ is $associated\ with$ a pair $(\mu, d)$ due to Lemma \ref{lem5}.
\begin{lem}\label{lem11}
Let $W$ denote an irreducible $T$-module  associated with a pair $(\mu, d)$, where $(\mu,d)\in  \Upsilon$. The following  {\rm{(i), (ii)}} hold.
\begin{itemize}
 \item[\rm(i)]  $W$ is thin; $W$ has endpoint $m-d$, diameter $2d+1$ and ${\rm dim}(W)=2d+2$.
\item[\rm(ii)]  The isomorphism class of $W$  depends only on the associated pair $(\mu, d)$.
\end{itemize}
\end{lem}
\begin{proof}
 (i) Observe that $W$ is thin by Lemma \ref{lem4}(i) and Lemma \ref{lem5}(iii).
It follows from  Lemma \ref{lem5}(iv) that $W$ has endpoint $m-d\ (0\leq d\leq m)$ and diameter $2d+1$. This together with thinness of $W$ immediately implies ${\rm dim}(W)=2d+2$.

 (ii) Immediate from Lemma \ref{lem4}(iii) and Lemma \ref{lem5}(i)--(iii).
\end{proof}
\begin{lem}\label{lem011}
There exist some irreducible $T$-modules  associated with a pair $(\mu, d)$ if and only if  $(\mu,d)\in  \Upsilon$.
\end{lem}
\begin{proof}
 By Lemma \ref{lem5}, it is easy to see that there exist some irreducible $T$-modules associated with a pair $(\mu, d)$ if and only if there exist some irreducible $\mathscr{T}$-modules having dual endpoint $\mu$ and diameter $d$, where $(\mu, d)\in \Upsilon$. From this fact and Lemma \ref{lem4}(iv), the result follows.
\end{proof}
In view of Lemmas \ref{lem11} and \ref{lem011}, let $W_{(\mu,d)}$ denote an irreducible $T$-module associated with the pair $(\mu,d)$ for $(\mu, d)\in \Upsilon$. Write the standard module $V:=\mathbb{C}^X$  as an orthogonal direct sum of irreducible $T$-modules, and we use  $mult(\mu,d)$ to denote the multiplicity of $W_{(\mu,d)}$ in this orthogonal direct sum.

By a {\it {homogeneous component}} of $V$, we mean  a  nonzero subspace of $V$ spanned by the  irreducible $T$-modules that are isomorphic. For a given $(\mu, d)\in \Upsilon$, we define $\mathcal{W}_{(\mu,d)}$ to be the subspace of $V$ spanned by the irreducible $T$-modules that are isomorphic to $W_{(\mu,d)}$. Clearly, $\mathcal{W}_{(\mu,d)}$ is a homogeneous component of $V$. Moreover, we have
\begin{align}\label{eq01}
V=\sum_{(\mu, d)\in \Upsilon}\mathcal{W}_{(\mu,d)}\ \ \ \  \ (\text {orthogonal\ direct\ sum}).
\end{align}

We below introduce some notation to describe our results explicitly. For two square matrices $A\in \mathbb{C}^{n\times n}$ and $B\in \mathbb{C}^{m\times m}$, we use  $A\oplus B\in \mathbb{C}^{(n+m)\times (n+m)}$ to denote  their direct sum. And we use $k\odot A$ to denote the iterated direct sum $\oplus^{k}_{i=1}A$  for an integer $k\geq 2$.  Furthermore, we use $\lceil a \rceil$ (resp. $\lfloor a \rfloor$) to denote the
minimal integer greater (resp. less) than or equal to $a$.
\begin{thm}\label{thm01}
The algebra $T$ is isomorphic to
\begin{equation}\label{eq02}
\bigoplus^{m}_{d=0}(m-d-\lceil\frac{m-d}{2}\rceil+1)\odot\mathbb{C}^{(2d+2)\times (2d+2)}.
\end{equation}
\end{thm}
\begin{proof}
From Lemmas \ref{lem11}, \ref{lem011} and \cite[pp. 96--98]{bek}, we have that there exists a unitary $2{2m+1\choose m}\times 2{2m+1\choose m}$ matrix $U$, whose
columns consist of
appropriate orthonormal bases of all the homogeneous components of $V$ in \eqref{eq01},
such that $\overline{U}^{\rm t}TU$ consists of all block-diagonal matrices:
\begin{equation}\label{eq11}
\bigoplus_{(\mu,d)\in \Upsilon}mult(\mu,d)\odot B_{(\mu,d)},
\end{equation}
where $B_{(\mu,d)}\in \mathbb{C}^{(2d+2)\times (2d+2)}$ and  $mult(\mu,d)$ denotes the multiplicity of an irreducible $T$-modules associated with a pair $(\mu,d)$, where $(\mu, d)\in \Upsilon$.
Then using the inequality  $\lceil\frac{m-d}{2}\rceil\leq \mu\leq m-d$ from \eqref{eq15} and deleting copies of blocks in \eqref{eq11}, we obtain \eqref{eq02}.
\end{proof}
\begin{cor}\label{lem21}
The dimension of $T$ is $4{m+4\choose 4}$.
\end{cor}
\begin{proof}
By Theorem \ref{thm01}, we have
\begin{align*}
{\rm dim}(T)&=\sum^m_{d=0}(m-d-\lceil\frac{m-d}{2}\rceil+1)(2d+2)^2
=4{m+4\choose 4},
\end{align*}
where the second equality holds by induction on $m$.
\end{proof}
\begin{lem}
The center of $T$ has dimension $\lfloor\frac{(m+2)^2}{4}\rfloor$.
\end{lem}
\begin{proof}
The dimension of the center of $T$ is exactly the number of isomorphism classes for irreducible $T$-modules. This number is  clearly $|\Upsilon|$ that equals $\frac{(m+1)(m+3)}{4}$ if $m$ is odd and $\frac{(m+2)^2}{4}$ if $m$ is even.
\end{proof}
\begin{thm}\label{thm2}
The  algebras {\rm $\mathcal{A}$}  and $T$ coincide.
\end{thm}
\begin{proof}
On the one hand, we have  $T\subseteq \mathcal{A}$
by Lemma \ref{lem3}. On the other hand, we have  $\text{dim}(\mathcal{A})=\text{dim}(T)=4{m+4\choose 4}$ by \eqref{eq16} and Corollary \ref{lem21}. So the result holds.
\end{proof}
\begin{cor}\label{cor2}
All the matrices of four types from \eqref{eq13}--\eqref{eq23}
give  a basis of   $T$.
\end{cor}
\begin{proof}
Immediate form Theorems \ref{thm1} and \ref{thm2}.
\end{proof}
\section*{Appendix}
Recall the Odd graph $O_{m+1}$ with vertex set $\mathscr{X}$.
We choose  $x_0:=\{1,2,\ldots,m\}\in \mathscr{X}$  as the base vertex since $O_{m+1}$ is distance-transitive. Let $\mathscr{A}:=\mathscr{A}(x_0)$ and $\mathscr{T}:=\mathscr{T}(x_0)$ denote the corresponding centralizer algebra  and Terwilliger algebra of $O_{m+1}$, respectively. In this appendix, we display some results about  $\mathscr{A}$ and $\mathscr{T}$, and most of these results are important and necessary for our discussions on the graph $2.O_{m+1}$ in the previous sections.

To each ordered pair $(y,z)\in \mathscr{X}\times \mathscr{X}$,
we similarly define the four-tuple: $\varrho(y,z):=(i,j,t,p)$ as \eqref{equ2}, and define the set
\begin{align}\label{eq25}
\mathcal{I}_m=\big\{(i,j,t,p)\mid \varrho(y,z)=(i,j,t,p),\ \ (y,z)\in  \mathscr{X}\times \mathscr{X}\big\}.
\end{align}
For each $(i,j,t,p)\in \mathcal{I}_m$, we further define the associated set \text{\ \ \ \ \ \ \ \ \ \  \ \ \ \ \ \ \ \ \ \ \ \ \ \ \  \ \ \ \ \ \ \ \ \ \ \ \ \ \ \ \ \ \ \ \ \ \ \ \ \ \ \ \ \ }
\ \text{\ \ \ \ \ \ \ \ \ \ \ \ \ \ \ \ \ \ \ \ \ \ \ \ \ \ \ \ \ \ \ \  }   $X_{(i,j,t,p)}=\{(y,z)\in  \mathscr{X}\times \mathscr{X}\mid \varrho(y,z)=(i,j,t,p)\}$.

Let ${\rm Aut}(O_{m+1})$ denote the automorphism group of $O_{m+1}$ and ${\rm Aut}_{x_0}(O_{m+1})$  the corresponding stabilizer of $x_0$. By \cite[p. 260]{bcn}, ${\rm Aut}(O_{m+1})$ is (up to isomorphic) sym$(S)$ and  hence ${\rm Aut}_{x_0}(O_{m+1})$ is (up to isomorphic) sym$(x_0)\times\text{sym}(S-x_0)$.
\begin{pro}\label{pro6}
the sets $X_{(i,j,t,p)}, (i,j,t,p)\in \mathcal{I}_m$, are the orbits of $\mathscr{X}\times \mathscr{X}$ under the\ action\ of  ${\rm Aut}_{x_0}(O_{m+1}).$
\end{pro}
\begin{proof}
For a given $(i,j,t,p)\in \mathcal{I}_m$, let $(y,z)\in X_{(i,j,t,p)}$ be associated with it: $\varrho(y,z)=(i,j,t,p)$. Pick any $\sigma \in {\rm Aut}_{x_0}(O_{m+1})$. It is easy to see  that  $\varrho(\sigma y,\sigma z)=(i,j,t,p)$ by the definitions of $i,j,t$ and $p$. This implies that $(\sigma y,\sigma z)\in X_{(i,j,t,p)}$.

To show that ${\rm Aut}_{x_0}(O_{m+1})$ acts transitively on $X_{(i,j,t,p)}$ for each given $(i,j,t,p)\in \mathcal{I}_m$, it suffices to show that for any pair $(y,z)$ satisfying $\varrho(y,z)=(i,j,t,p)$ there is an automorphism $\sigma\in  {\rm Aut}_{x_0}(O_{m+1})$ such that
$(\sigma y,\sigma z)$ is a fixed pair that  depends only on $i,j,t$ and $p$. By $\varrho(y,z)=(i,j,t,p)$, we have
$|x_0\cap y|=i,\ |x_0\cap z|=j,\ |y\cap z|=t\ \text{and}\ |x_0\cap y\cap z|=p$. Let $A=x_0\cap y\cap z,\ B=x_0\cap y-x_0\cap y\cap z,\ C=x_0\cap z-x_0\cap y\cap z,\ D=y\cap z-x_0\cap y\cap z,\ E=x_0-y\cup z,\ F=y-x_0\cup z,\
G=z-x_0\cup y\ \text{and}\ H=S-x_0\cup y\cup z$.  We then have $|A|=p,\ |B|=i-p,\ |C|=j-p,\ |D|=t-p,\ |E|=m-i-j+p,\ |F|=m-i-t+p,\ |G|=m-j-t+p\ \text{and}\  |H|=i+j+t-p-m+1$. One can readily verify that $x_0=A\cup B\cup C\cup E$,\
$y=A\cup B\cup D\cup F$ and $z=A\cup C\cup D\cup G$.
Pick an automorphism $\sigma\in  {\rm Aut}_{x_0}(O_{m+1})$  such that under the action of $\sigma$:
$A\rightarrow\{1,\ldots,p\}$, $B\rightarrow\{p+1,\ldots,i\}$, $C\rightarrow\{i+1,\ldots,i+j-p\}$, $E\rightarrow\{i+j-p+1,\ldots,m\}$, $D\rightarrow\{m+1,\ldots,m+t-p\}$, $F\rightarrow\{m+t-p+1,\ldots,2m-i\}$, $G\rightarrow\{2m-i+1,\ldots,3m-i-j-t+p\}$, $H\rightarrow\{3m-i-j-t+p+1,\ldots,2m+1\}$.
Thus, we obtain that $\sigma(x_0)=\sigma(A\cup B\cup C\cup E)=\{1,2,\ldots,m\}$, $\sigma(y)=\sigma(A\cup B\cup D\cup F)=\{1,\ldots,i,m+1,\ldots,2m-i\}$ and $\sigma(z)=\sigma(A\cup C\cup D\cup G)=\{1,\ldots,p,i+1,\ldots,i+j-p,m+1,\ldots,m+t-p,2m-i+1,\ldots,3m-i-j-t+p\}$.

This completes the proof.
\end{proof}

 For each $(i,j,t,p)\in \mathcal{I}_m$, the above orbits naturally make us  define some  matrices of 0s and 1s in ${\rm{Mat}}_\mathscr{X}(\mathbb{C})$:
\begin{equation}
(M^{t,p}_{i,j})_{yz}=\left\{\begin{array}{ll} 1 &\text{if}\ (y,z)\in X_{(i,j,t,p)},\\[0.2cm]
 0 &\text{otherwise } \end{array}\right.
\ \ (y,z\in \mathscr{X}).
\end{equation}
It is easy to see that {\bf all\ the\ matrices\  $M^{t,p}_{i,j},\ (i,j,t,p)\in \mathcal{I}_m$,\ give\ a\ basis\ of}\ $\mathscr{A}$.

Next, we will compute dim($\mathscr{A}$) which is just the cardinality of $\mathcal{I}_m$.  To do this, we need the fact from \cite[Corollary 3.7]{klw} that\
$dim(\mathscr{T})={m+4\choose 4}\ for\ m\geq 1.$
\begin{pro}\label{pro5}
 The set $\mathcal{I}_m$, for $m\geq 1$, is
\begin{align}\label{eq06}
\big\{(i,j,t,p)\mid\ &0\leq i,j\leq m,\  \mathrm{max}\{i+j-m,m-1-i-j\}\leq t\leq m-|i-j|, \nonumber\\
&\mathrm{max}\{0,i+j-m,i+t-m,j+t-m\}\leq p\leq \\
&\ \ \ \ \ \ \ \ \ \  \ \ \ \ \ \ \  \ \ \ \ \ \ \ \ \ \  \ \ \ \ \ \ \ \ \ \ \ \ \ \ \ \ \ \ \ \ \ \ \ \ \mathrm{min}\{i,j,t,i+j+t+1-m\}\big\}.\nonumber
\end{align}
Moreover, the cardinality of $\mathcal{I}_m$ is
${m+4\choose 4}$.
\end{pro}
\begin{proof}
Denote  the set \eqref{eq06} by $\mathcal{I}'_m$. Below,   we first  show $\mathcal{I}_m\subseteq\mathcal{I}'_m$.
For each given $(i,j,t,p)\in \mathcal{I}_m$, observe that
\begin{align}\label{equ8}
0\leq p\leq i,j,t\leq m
\end{align}
by \eqref{equ2} and  \eqref{eq25}.
Let $y,z\in \mathscr{X}$ and let $\varrho(y,z)=(i,j,t,p)$, where $i=|x_0\cap y|,\ j=|x_0\cap z|,\  t=|y\cap z|,\  p=|x_0\cap y\cap z|$. Then we have
\begin{align*}
\partial_H(x_0,y)=2m-2i,\ \partial_H(x_0,z)=2m-2j,\ \partial_H(y,z)=2m-2t,
\end{align*}
where $\partial_H(u,v):=2m-2|u\cap v|$ denotes the {\it Hamming\ distance} between  $u$ and $v\ (u,v\in \mathscr{X})$. From the two inequalities $\partial_H(y,z)\leq \partial_H(x_0,y)+\partial_H(x_0,z)$  and $|\partial_H(x_0,y)-\partial_H(x_0,z)|\leq \partial_H(y,z)$, it follows that
$i+j-m\leq t\leq m-|i-j|$.
Moreover,
it is easy to see that $|y\cup z|\leq |S|-(|x_0|-i-j)$ if $i+j\leq m-1$, which implies
$m-1-i-j\leq t$  for $i+j\leq m-1$.
Combine the above two inequalities involving $t$ to obtain
\begin{align}\label{equ9}
\mathrm{max}\{i+j-m,m-1-i-j\}\leq t\leq m-|i-j|.
\end{align}

Furthermore, by using the three inequalities:\  \  $i-p=|x_0\cap y-z|\leq |x_0-z|=m-j$,\\
\text{\ \ \ \ \ \ \ \ \ \ \ \ \ \ \ \ \ \ \ \ \ \ \ \ \ \ \ \ \ \ \ \ \ \ \ \ \ \ \ \ \ \ \ \ \ \ \ \ \ \ \ \ \ \ \ \ \ \ \ \ \ \ \ \ } $i-p=|x_0\cap y-z|\leq |y-z|=m-t,$\\
\text{\ \ \ \ \ \ \ \ \ \ \ \ \ \ \ \ \ \ \ \ \ \ \ \ \ \ \ \ \ \ \ \ \ \ \ \ \ \ \ \ \ \ \ \ \ \ \ \ \ \ \ \ \ \ \ \ \ \ \ \ \ \ \ \ } $j-p=|x_0\cap z-y|\leq |z-y|=m-t$\\
we can obtain $\mathrm{max}\{i+j-m,i+t-m,j+t-m\}\leq p$.
Moreover, we have $p\leq i+j+t+1-m$ since $|x_0\cup y\cup z|\leq 2m+1$. Combine the above two inequalities involving $p$ to obtain
\begin{align}\label{equ12}
\mathrm{max}\{i+j-m,i+t-m,j+t-m\}\leq p\leq i+j+t+1-m.
\end{align}
From \eqref{equ8}--\eqref{equ12}, we easily obtain $(i,j,t,p)\in \mathcal{I}'_m$
 and therefore we have\
$\mathcal{I}_m\subseteq\mathcal{I}'_m.$

Next, we shall show
\begin{equation}\label{equ13}
|\mathcal{I}'_m|={m+4\choose 4}\ \ \text{for}\ \ m\geq 1
\end{equation}
by induction on $m$. For $m=1$, one can readily verify that   \eqref{equ13} holds by computing
$|\mathcal{I}'_1|=5$. We assume
that \eqref{equ13} holds for $m=k-1\ (k\geq 2)$, that is,
\begin{equation}\label{equ14}
|\mathcal{I}'_{k-1}|={k+3\choose 4}.
\end{equation}
To compute the $|\mathcal{I}'_k|$, we define some subsets of $\mathcal{I}'_k$ as follows: for each $i\ (0\leq i\leq k)$ and each $l\ (1\leq l\leq i)$, let\ \ \ \ \ \ \
$\mathcal{B}_{i,i}=\{(i,i,t,p)\mid(i,i,t,p)\in \mathcal{I}'_{k}\}$, \\
\text{\ \ \ \ \ \ \ \ \ \ \ \ \ \ \ \ \ \ \ \ \ \ \ \ \ \ \ \ \ \ \ \ \ \  \ }$\mathcal{B}_{i,i-l}=\{(i,i-l,t,p)\mid(i,i-l,t,p)\in \mathcal{I}'_k\}$,\\
\text{\ \ \ \ \ \ \ \ \ \ \ \ \ \ \ \ \ \ \ \ \ \ \ \ \ \ \ \ \ \ \ \ \ \  \ }$\mathcal{B}_{i-l,i}=\{(i-l,i,t,p)\mid(i-l,i,t,p)\in \mathcal{I}'_k\}.$\\
Observe that $|\mathcal{B}_{i,i-l}|=|\mathcal{B}_{i-l,i}|$ and all these  subsets are pairwise disjoint. Hence, we have
\begin{align}\label{equ15}
|\mathcal{I}'_k|=\sum^k_{i=0}|\mathcal{B}_{i,i}|+\sum^k_{i=1}\sum^i_{l=1}|\mathcal{B}_{i,i-l}|+
\sum^k_{i=1}\sum^i_{l=1}|\mathcal{B}_{i-l,i}|.
\end{align}

By the definition of $\mathcal{I}'_k$ and of $\mathcal{B}_{i,i}\ (0\leq i\leq k)$, it is not difficult to compute
\begin{align*}
|\mathcal{B}_{i,i}|=\left\{\begin{array}{ll} (i+1)(i+2) &\text{if $0\leq i\leq \lfloor\frac{k-1}{2}\rfloor$},\\[0.1cm]
(k+1-i)^2 &\text{if $\lfloor\frac{k-1}{2}\rfloor+1\leq i\leq k$}.
 \end{array}\right.
\end{align*}
This implies
\begin{align}\label{equ18}
\sum^k_{i=0}|\mathcal{B}_{i,i}|=\left\{\begin{array}{ll} \frac{(k+1)(k+3)(2k+7)}{24} &\text{if $k$ is odd},\\[0.2cm]
\frac{(k+2)(k+4)(2k+3)}{24} &\text{if $k$ is even}.
 \end{array}\right.
\end{align}
For each $i\ (1\leq i\leq k)$ and each $l\ (1\leq l\leq i)$,
apply the definition of $\mathcal{B}_{i,i-l}$ and of $\mathcal{B}_{i-l,i}$ to get that  both
\begin{align*}
&|\mathcal{B}_{i,i-l}|=|\{(i-1,i-l,t,p)\mid(i-1,i-l,t,p)\in \mathcal{I}'_{k-1}\}|,\ \ \ \ \ \\
&|\mathcal{B}_{i-l,i}|=|\{(i-l,i-1,t,p)\mid(i-l,i-1,t,p)\in \mathcal{I}'_{k-1}\}|.
\end{align*}
It follows from the above two equations that
\begin{align}\label{equ17}
\sum^k_{i=1}\sum^i_{l=1}|\mathcal{B}_{i,i-l}|+
\sum^k_{i=1}\sum^i_{l=1}|\mathcal{B}_{i-l,i}|=|\{(i,i,t,p)\mid(i,i,t,p)\in\mathcal{I}'_{k-1}\}|+|\mathcal{I}'_{k-1}|.
\end{align}
Note that the value of $|\{(i,i,t,p)\mid(i,i,t,p)\in\mathcal{I}'_{k-1}\}|$ can be computed directly by replacing $k$ by $k-1$ in \eqref{equ18}. Combine \eqref{equ14}--\eqref{equ17} to obtain
\begin{align}
|\mathcal{I}'_{k}|
&=\sum^k_{i=0}|\mathcal{B}_{i,i}|+|\{(i,i,t,p)\mid(i,i,t,p)\in\mathcal{I}'_{k-1}\}|+
|\mathcal{I}'_{k-1}| \nonumber\\[0.1cm]
&=\sum^k_{i=0}|\mathcal{B}_{i,i}|+\sum^{k-1}_{i=0}|\mathcal{B}_{i,i}|+|\mathcal{I}'_{k-1}| \nonumber\\[0.1cm]
&=\frac{(k+1)(k+2)(k+3)}{6}+{k+3\choose 4}={k+4\choose 4},\nonumber\
\end{align}
as desired. Thus the equation \eqref{equ13} holds.

Now, we claim that the set $\mathcal{I}_{m}$ is the same as the set $\mathcal{I}'_{m}$. Suppose $\mathcal{I}_{m}$ is
a proper subset of $\mathcal{I}'_{m}$ for a contradiction. Since $\mathscr{T}$ is a  subalgebra of $\mathscr{A}$, then by \eqref{equ13} we have
$$\text{dim}(\mathscr{T})\leq \text{dim}(\mathscr{A})=|\mathcal{I}_{m}|<|\mathcal{I}'_{m}|={m+4\choose 4}.$$
This clearly contradicts $\text{dim}(\mathscr{T})={m+4\choose 4}$ and hence our claim holds.

This completes the proof.
\end{proof}
The above discussions give ${\rm dim}(\mathscr{T})={\rm dim}(\mathscr{A})={m+4\choose 4}$. This fact together with  $\mathscr{T}\subseteq\mathscr{A}$ implies that {\bf the\ two\ algebras\ $\mathscr{A}$ and $\mathscr{T}$ coincide}.

Recall the standard  module $\mathscr{V}=\mathbb{C}^\mathscr{X}$. In what follows, we focus on the irreducible $\mathscr{T}$-modules and give the proof of Lemma \ref{lem4}(iv). Assume $m\geq 3$. Let $\mathscr{W}$ denote an irreducible $\mathscr{T}$-module with dual endpoint $\mu$ and diameter $d$, where $(\mu,d)\in  \Upsilon$. Define $\mathscr{W}_{(\mu,d)}$ to be the subspace of $\mathscr{V}$ spanned by the irreducible $T$-modules that are isomorphic to $\mathscr{W}$. Then we obtain
\begin{align}\label{equ56}
\mathscr{V}=\sum_{(\mu, d)\in \Upsilon}\mathscr{W}_{(\mu,d)}\ \ \ \  \ (\text {orthogonal\ direct\ sum}).
\end{align}
\begin{proof}[\bfseries\rm\bf Proof of Lemma \ref{lem4}(iv)]
Lemma \ref{lem4}(iii) implies (iv) in one direction. We next prove (iv) in the other direction. Let $\Psi$ be the subset of  $ \Upsilon$ containing all dual endpoints and diameters that arise from irreducible $\mathscr{T}$-modules.  Suppose $\Psi$ is a proper subset of  $ \Upsilon$ for a contradiction. Then by using a similar argument used in the proof of Theorem \ref{thm01}, we can
obtain the following inequality
\begin{align*}
{\rm dim}(\mathscr{T})=\sum_{(\mu,d)\in \Psi}(d+1)^2<\sum_{(\mu,d)\in  \Upsilon}(d+1)^2
&=\sum^m_{d=0}(m-d-\lceil\frac{m-d}{2}\rceil+1)(d+1)^2\\
&={m+4\choose 4}.
\end{align*}
This contradicts ${\rm dim}(\mathscr{T})={m+4\choose 4}$. So $\Psi= \Upsilon$. Thus the result holds.
\end{proof}
From Lemma \ref{lem4},  it follows immediately that
{\bf the\ algebra\ $\mathscr{T}$\ is\ isomorphic\ to}
$\bigoplus\limits^{m}_{d=0}(m-d-\lceil\frac{m-d}{2}\rceil+1)\odot\mathbb{C}^{(d+1)\times (d+1)}.$

\section*{Acknowledgement}

This work is supported by the NSF of China (No. 11971146 and No. 12101175), the NSF of Hebei Province (No. A2019205089 and No. A2020403024).

\end{document}